\documentclass[a4paper,11pt]{amsart}

\usepackage{enumerate}
\usepackage{amsthm,amsmath,amsfonts,amssymb}
\usepackage{mathrsfs}
\usepackage[pdftex]{graphicx}
\usepackage[colorlinks=true,linkcolor=blue,urlcolor=blue, citecolor=blue]{hyperref}
\usepackage[draft,inline,marginclue]{fixme}
\usepackage{amscd}
\usepackage{xcolor}
\usepackage{ifpdf} 
\usepackage[all]{xy}

\usepackage{tikz}
\usetikzlibrary{matrix,arrows,calc}
\usepackage{aliascnt}
\usepackage{mathtools}

\usepackage{epstopdf}

\numberwithin{equation}{section}
\numberwithin{table}{section}

\newcommand{\overbar}[1]{\mkern 1.5mu\overline{\mkern-1.5mu#1\mkern-1.5mu}\mkern 1.5mu}

\newcommand{\mynewtheorem}[4]{
  \if\relax\detokenize{#3}\relax 
    \if\relax\detokenize{#4}\relax 
      \newtheorem{#1}{#2}
    \else
      \newtheorem{#1}{#2}[#4]
    \fi
  \else
    \newaliascnt{#1}{#3}
    \newtheorem{#1}[#1]{#2}
    \aliascntresetthe{#1}
  \fi
  \expandafter\def\csname #1autorefname\endcsname{#2}
}

\def\equationautorefname~#1\null{(#1)\null}

\theoremstyle{plain}
\mynewtheorem{thm}{Theorem}{}{section}
\mynewtheorem{lemma}{Lemma}{thm}{}
\mynewtheorem{rem}{Remark}{lemma}{}
\mynewtheorem{prop}{Proposition}{lemma}{}
\mynewtheorem{corollary}{Corollary}{lemma}{}

\theoremstyle{definition}

\newtheorem{exampleth}[thm]{Example}
\newenvironment{example}{\begin{exampleth}}{\hfill $\diamond$\\ \end{exampleth}}

\DeclareMathOperator{\Spec}{\operatorname{Spec}}

\DeclareMathOperator{\conv}{\operatorname{conv}}

\DeclareMathOperator{\GL}{\operatorname{GL}}

\DeclareMathOperator{\trop}{trop}

\DeclareMathOperator{\cl}{cl}
\DeclareMathOperator{\cone}{cone}
\DeclareMathOperator{\supp}{supp}

\newcommand{\cA}{{\mathcal A}}

\newcommand{\F}{\mathbb{F}}

\newcommand{\an}{{\mathrm{an}}}
\newcommand{\G}{\mathbb{G}}

\newcommand{\Bfrak}{\mathfrak{B}}

\newcommand{\ov}{\overline}

\title{On endomorphisms of Arrangement Complements}
\author{\c{S}evda Kurul and Annette Werner}
\address{Institut f\"ur Mathematik,
Goethe-Universit\"at Frankfurt, Robert-Mayer-Str. 6-8, 60325 Frankfurt am Main,Germany}
\email{\{kurul\},\{werner\}@math.uni-frankfurt.de}
\date{\today}                                        

\keywords{}
\subjclass[2010]{}

\begin{document}
\begin{abstract}
Let $\Omega$ be the complement of a connected, essential  hyperplane arrangement. We prove that every dominant endomorphism of $\Omega$ extends to an endomorphism of the tropical compactification $X$ of $\Omega$ associated to the Bergman fan structure  on the tropical variety $\trop(\Omega)$. This generalizes a result in \cite{rtw}, which states that every automorphism of Drinfeld's half-space over a finite field $\mathbb{F}_q$ extends to an automorphism of the successive blow-up of projective space at all $\mathbb{F}_q$-rational linear subspaces. This successive blow-up is in fact the minimal wonderful compactification by de Concini and Procesi, which coincides with $X$ by results of Feichtner and Sturmfels \cite{fs}. Whereas the proof in \cite{rtw} is based on Berkovich analytic geometry over the trivially valued finite ground field, the  generalization proved in the present paper relies on matroids and tropical geometry. 
\end{abstract}

\maketitle
\small
\centerline{\bf MSC(2010): 14T05, 52C35}
\normalsize

\section{Introduction}\label{sec:intro}

Let $\cA$ be a connected, essential arrangement of hyperplanes over an arbitrary field $K$, and let $\Omega_\cA$ be the complement of the arrangement $\cA$ in projective space.  By  $X_{vc}(\cA)$ we denote the visible contour compactification of $\Omega_\cA$. It is associated to the Bergman fan structure on the tropicalization $\trop(\Omega_\cA)$ in the sense of Tevelev \cite{Tev}. Our main result \autoref{mainthm} states that every dominant endomorphism of $\Omega_\cA$ extends to an endomorphism of its visible contour compactification $X_{vc}(\cA).$ As a corollary we show in \autoref{finite} that every dominant endomorphism of $\Omega_\cA$ is finite. Feichtner and Sturmfels \cite{fs} have provided conditions under which  the visible contour compactification coincides with the minimal wonderful compactification of $\Omega_\cA$ defined by de Concini and Procesi.

This coincidence occurs for example if $K= \F_q$  is a finite field and $\cA$ is the full arrangement of all $\F_q$-rational hyperplanes in projective space. Then the complement $\Omega_\cA$ is Drinfeld's  half-space over $\F_q$.  It was shown in \cite[Theorem 1.1]{rtw} that for this arrangement every automorphism of $\Omega_\cA$ extends to an automorphism of the ambient projective space $\mathbb{P}_{\F_q}^d$, i.e. it is given by an element in $PGL(d, \F_q)$. An important  step in the proof  is the extension of an automorphism of $\Omega_\cA$ to an automorphism of the successive blow-up $X_{wnd}(\cA)$ of $\mathbb{P}^d_{\F_q}$ at all $\F_q$-rational linear subspaces, which is achieved by using Berkovich analytic geometry over the trivially valued field $\F_q$. In the present paper, see \autoref{cor:drinfeld}, we give an alternative proof of this step without using analytic geometry. Instead we use techniques from tropical geometry and matroid theory.
Our alternative approach can then be generalized to arbitrary essential and connected hyperplane arrangements over any ground field. 
To be more precise, the proof of \cite[Theorem 1.1]{rtw} relies on the  fact that every automorphism of $\Omega_\cA$ restricts to an automorphism of a suitable skeleton of $\Omega^{\an}_\cA$. In order to show that this restriction preserves the fan structure on the skeleton, it is proved in \cite[Lemma 2.2]{rtw} that distinct maximal cones in the skeleton span distinct linear spaces. In the present paper, we prove a tropical avatar of this result in Theorem \ref{thm:cones}, which states that for a loopfree matroid $M$ distinct maximal cones in the Bergman fan span distinct linear spaces. 

Note that the Drinfeld half-space is the only hyperplane complement over $\mathbb{F}_q$ with automorphism group $PGL(d,\mathbb{F}_q).$
Therefore we cannot expect that the second step in the proof of \cite[Theorem 1.1]{rtw}, which is a descent from $X_{wnd}(\cA)$ to projective space, can be generalized to other arrangements. Langer \cite{langer} shows in a recent paper that every separable dominant endomorphism of Drinfeld's half-space is in fact an automorphism, but that there exist inseparable endomorphisms which do not have this property and which therefore cannot be extended to automorphisms of the full projective space.

The outline of this paper is as follows. We start with some basic definitions and properties of hyperplane arrangements in \autoref{sec:background}, together with a brief account on the compactifications of arrangement complements introduced in \cite{Tev}, \cite{DP} and \cite{kapranov}. In \autoref{sec:matroidtheory} we introduce the necessary definitions from matroid theory, and we describe different fan structures on the tropical linear space of a matroid, ranging from the coarsest (the Bergman fan) to the finest (the fine subdivision). In \autoref{sec:Bergman} we prove that distinct cones of the Bergman fan span distinct linear spaces. Then we  prove in \autoref{sec:extend} that every dominant endomorphism of a connected arrangement complement extends to an endomorphism of its visible contour compactification and hence is finite. In particular, the automorphism group of the arrangement complement is a subgroup of the automorphism group of its visible contour compactification.

{\bf Acknowledgements:} We want to thank the referee for his or her remarks on this paper. We also thank Adrian Langer for asking if the results of \cite{rtw} may be generalized to endomorphisms and for his hospitality during a visit of the second author to Warsaw. We are grateful to Kristin Shaw for many useful discussions. We also thank Diane Maclagan for her very helpful comments at various stages of this project. Research on this paper was supported by DFG grant WE-4279/7.

\section{Hyperplane complements and compactifications}\label{sec:background}

\subsection*{Hyperplane complements} Fix any ground field $K$ and a vector space $V$ of dimension $d+1$ over $K$. A set  $\cA= \{H_0, \ldots, H_n\}$ of $n+1$ linear hyperplanes in $V$ is called a \emph{hyperplane arrangement} over $K$.  It is called \emph{essential}, if the intersection $\bigcap_{i=0}^{n} H_i = \{ 0 \}.$ 

\begin{example} One important family of essential hyperplane arrangements is given by the (essential) braid arrangements $A_{n}$ for $n>1$. Here we consider the hyperplanes $H_i = V(x_i)$ for $i \in \{ 0, \ldots, n-1\}$ and $H_{ij} = V(x_i - x_j)$ for $i,j \in \{0, \ldots ,n-1 \}$ and $i<j$ in $n$-space. This arrangement is in fact the quotient of the finite reflection arrangement associated to a type $A$ root system after dividing by the lineality space. The complement of $A_{n}$ in $\mathbb{P}^{n-1}_{\mathbb{C}}$ is isomorphic to the moduli space $M_{0,n+2}$ of $n+2$ pointed curves of genus $0.$ 
\end{example}

\subsection*{Tropicalizations} From now on we assume that  $\cA = \{H_0, \ldots, H_n\}$ is an essential arrangement of hyperplanes in $V$. We denote by $\Omega_\cA$ the complement $\mathbb{P}(V) - \cA$ endowed with the reduced induced structure, where $\mathbb{P}(V) = \mbox{Proj Sym} V^\ast$ is the projective space of lines in $V$.  Then $\Omega_\cA$ is an integral affine $K$-scheme. Furthermore, let  $T$  be the standard torus in $\mathbb{P}^n_K$ 
which is the complement of all coordinate hyperplanes.  Let $l_i$ be an element in the dual space $V^\ast$ such that $H_i$ is the kernel of $l_i$. 
Then the morphism  $$j: \Omega_\cA \to T, \quad x \mapsto [l_0(x): l_1(x): \ldots : l_n(x)]$$
is a closed immersion. 

Moreover, the  multiplicative group $\mathcal{O}(\Omega_\cA)^*/K^*$ is generated by the classes $\left[\frac{l_i}{l_0}\right]$ of elements $\frac{l_i}{l_0}$ of $\mathcal{O}(\Omega_\cA)^*,$ for $i$ in $\{ 1, \ldots, n\}$. Hence $T$ can be identified with the intrinsic torus of the very affine variety $\Omega_\cA$.
By $\trop(\Omega_\cA)$ we denote the associated tropicalization where the ground field $K$ is taken with the trivial absolute value, see  \cite[Section 4.1]{tropbook}.
We denote by $N$ the cocharacter group of $T$.

The hyperplane arrangement $\cA$  gives rise to a matroid $M_\cA$ on the ground set $\{ 0, \ldots, n \}$, whose independent sets correspond to the linear independent subsets of $\{ l_0, \ldots, l_n \}$ in $V^\ast$. A good introduction to matroids can be found in \cite{Ox}.

We call the arrangement $\cA$ \emph{connected} if its associated matroid $M_\cA$ is connected. This means that $\cA$ cannot be decomposed as a product of strictly smaller arrangements.
The lattice of flats of the matroid $\mathcal{L}(M_\cA)$ of $M_\cA$ is just the intersection lattice of $\cA$ and the rank function $r$ on $M_\cA$ is given by the codimension of the corresponding intersection. The lattice $\mathcal{L}(M_\cA)$ is partially ordered by reverse inclusion. In fact, $\mathcal{L}(M_\cA)$ is a geometric lattice with minimal element $\hat 0,$ which corresponds to the empty intersection, hence to the ambient space of the arrangement. 

In general, a loopfree matroid $M$ on a finite set $E(M) = \{ 0,1, \dotsc ,n \}$ gives rise to a tropicalization in the following way. Write $\mathbb{R}^{n+1} / \mathbb{R} \cdot \textbf{1}$ for the quotient space $\mathbb{R}^{n+1}  / \mathbb{R} \cdot (1, \ldots, 1).$ Then
the \emph{tropical linear space} $\trop(M)$  of a loopfree matroid M is the set of vectors $v = (v_0, v_1, \dotsc , v_n) \in \mathbb{R}^{n+1}$ such that, for every circuit $C$ of $M$, the minimum of the numbers $v_i$ is attained at least twice as $i$ ranges over $C$. If $v \in \trop(M)$ then $v + \lambda \cdot \textbf{1} \in \trop(M)$ for any $\lambda \in \mathbb{R}$, so we regard it as a subset of $\mathbb{R}^{n+1}/\mathbb{R} \cdot \textbf{1}$.

Note that for any essential arrangement of hyperplanes $\cA$, the tropicalization $\trop(\Omega_\cA)$ which we defined previously coincides with the tropical linear space $\trop(M_\cA)$ by \cite[Proposition 4.1.6]{tropbook}, since the ideal of $j(\Omega_\cA)$   as a subvariety of $T$ is generated by a system of linear equations given by the  circuits of the matroid $M_\cA.$ In fact, for a linear form $l = \sum a_ix_i \in I$ we define its support by $\supp(l) = \{i: a_i \neq 0 \}.$ Then the linear forms $l_C$ of $I,$ such that $\supp(l_C)$ is a circuit of $M_\cA$, form a tropical basis.

\begin{example} For the braid arrangement $A_3$ in $\mathbb{P}^2_K$ the set of circuits of $M_{A_3}$ is just the index sets of the minimal dependent sets of column vectors of the matrix \[ B = \begin{pmatrix}
 1 & 0 & 0 & 1 & 1 & 0\\
 0 & 1 & 0 & -1 & 0 & 1\\
 0 & 0 & 1 & 0 & -1 & -1
\end{pmatrix}.\] If we index the columns of $B$ from $0$ to $5$ then the set of circuits of $M_{A_3}$ is given by $$ \{ \{0,1,3\}, \{0,2,4\}, \{1,2,5\}, \{3,4,5\},  \{0,1,4,5\},\{0,2,3,5\},\{1,2,3,4\}\}.$$ The complement $\mathbb{P}^2_K - A_{3}$ is identified with the very affine variety $V(I)$ in the torus $\G_m^6/\G_m$ given by the ideal $$I=(x_0 - x_1 -x_3, x_0 - x_2 -x_4, x_1 - x_2 - x_5, x_4 -x_3 -x_5)$$ in $K[x_0^{\pm}, \ldots, x_5^{\pm}].$
\end{example}

\subsection*{Compactifications} Let $K$ be any field, and let $T$ be a split torus over $K$ with cocharacter group $N$. 
For every fan $\Sigma$ in $N_{\mathbb{R}} = N \otimes_{\mathbb{Z}} \mathbb{R},$ we denote by $Y_{\Sigma}$ the normal toric $K$-variety associated to $\Sigma$ with dense torus $T$. 

Let us recall some results by Tevelev over an algebraically closed ground field $K$. In \cite[Proposition 2.3]{Tev} Tevelev shows that in this case the closure $\overline{\Omega}_\cA$ of $\Omega_\cA$ in the (not necessarily complete) toric variety $Y_\Sigma$ is complete if and only if the support of $\Sigma$ contains $\trop(\Omega_\cA)$. In particular every choice of a fan structure on $\trop(\Omega_\cA)$ gives rise to a compactification of the arrangement complement $\Omega_\cA,$ even if the toric variety itself is not complete. 

If $K$ is an arbitrary ground field with algebraic closure $\overline{K}$, the complement $\mathbb{P}_{\overbar{K}} - \cA$, where $\cA$ is regarded as an arrangement in $\mathbb{P}_{\overbar{K}}$, is the base change $\Omega_{\cA} \otimes_K \overbar{K} = \Omega_{\cA} \times_{\Spec(K)} \Spec(\overbar{K}).$  The tropicalizations $\trop(\Omega_\cA)$ and $\trop(\Omega_\cA \otimes_K \overline{K})$ coincide. Note that for any fan structure $\Sigma$ on $\trop(\Omega_{\cA})$, the base change of the closure $ \overline{\Omega}_{\cA}$ of $\Omega_\cA$ in the toric $K$-variety $Y_{\Sigma}$ coincides with the closure of $\Omega_{\cA} \otimes_K \overline{K}$ in $Y_{\Sigma} \otimes_K \overbar{K}$ which is proper. Hence  by faithfully flat descent, $\overline{\Omega}_{\cA}$ is also proper over $K$. 

Compactifications of subvarieties of tori obtained in this way are called \emph{tropical compactifications}, if the multiplication map $\mu: T \times \overline{\Omega}_{\cA} \to Y_{\Sigma}$ is faithfully flat. A subvariety of a torus is called \emph{sch\"{o}n} if the multiplication map for one (hence for any \cite[Theorem 1.4]{Tev}) tropical compactification is smooth. By \cite[Theorem 1.5]{Tev} $\Omega_{\cA}$ is sch\"{o}n for any connected, essential arrangement $\cA.$

Note that in general there is no canonical fan structure on the tropicalization of a very affine variety, so that there are different natural tropical compactifications.
Here we are mainly interested in the following two compactifications. The first one is obtained by taking the Bergman fan $\Bfrak(M_{\cA})$ on the tropicalization $\trop(\Omega_{\cA}).$ We denote the closure $\overline{\Omega}_{\cA} \subset Y_{\Bfrak(M_{\cA})}$ by $X_{vc}(\cA)$ and, following \cite{Tev}, we call it the \emph{visible contour compactification}. Over an algebraically closed field $K$ this is the visible contour compactification investigated by Kapranov \cite{kapranov}. Since $X_{vc}(\cA)$ is a tropical compactification, $X_{vc}(\cA)$ is smooth if and only if the compactifying toric variety $Y_{\Bfrak(M_{\cA})}$ is smooth. The Bergman fan $\Bfrak(M_{\cA})$ is not necessarily simplicial, therefore $X_{vc}(\cA)$ is not smooth in general. A family of examples can be found in \cite{DD}.

The second compactification of interest here is constructed by taking the minimal nested set fan $\Sigma_{min}(M_{\cA})$ as fan structure supported on $\trop(\Omega_{\cA}).$ We write $X_{wnd}(\cA)$ for the closure $\overline{\Omega}_{\cA} \subset Y_{\Sigma_{min}(M_{\cA})}$ and call it the \emph{wonderful compactification} of $\Omega_{\cA}.$ Feichtner and Sturmfels have shown in \cite{fs} that over an algebraically closed field $X_{wnd}(\cA)$ coincides with the minimal wonderful model of the arrangement complement introduced by de Concini and Procesi in \cite{DP}. The compactification $X_{wnd}(\cA)$ can also be obtained by iteratively blowing up the ambient projective space of $\cA$ along strict transforms of linear subspaces in increasing order of dimension. The boundary $X_{wnd}(\cA) \backslash \Omega_\cA$ is a divisor with normal crossings whose irreducible components are indexed by the elements of the so-called building set. A subset of boundary components intersect if and only if the corresponding subset of the building set forms a nested set. In \autoref{sec:matroidtheory} we give a formal definition of building sets, nested sets as well as of the fans $\Bfrak(M_{\cA})$ and $\Sigma_{min}(M_{\cA}).$

\section{Matroids and fan structures}\label{sec:matroidtheory}
Let us begin by recalling some facts on matroids. Let $M$ be a matroid on the finite set $E(M)$. We denote by $\mathcal{B}(M)$ the set of bases, by $\mathcal{C}(M)$ the set of circuits and by $\cl_M$ its closure operator.

\subsection*{Restriction and contraction of matroids}
For a subset $X \subset E(M)$ we define the \emph{restriction} of $M$ to $X$ as the matroid $M \vert_X$ on the ground set $X$ for which a subset of $X$ is independent if and only if it is independent in the original matroid $M.$  A subset $F$ of $X$ is a flat of $M \vert_X$ if and only if there is a flat $\widetilde{F}$ of $M$ such that $F = \widetilde{F} \cap X.$ 

 We define the \emph{contraction} of $M$ to $E(M) \backslash X$ as the matroid $M/X$ on the ground set $E(M) \backslash X$ whose independent sets are subsets $I$ of $E(M) \backslash X$ such that for some (equivalently, for any) basis $B_X$ of the restriction $M\vert_X$ the set  $I \cup B_X$ is an independent set of $M.$ A subset $F$ of $E(M) \backslash X$ is a flat of $M/X$ if and only if $F \cup X$ is a flat of $M.$

\begin{example} Let $\mathcal{A}$ be an essential arrangement of $n+1$ hyperplanes in a vector space $V,$ let $M(\mathcal{A})$ its associated matroid on $\{ 0,1, \ldots, n\}$ and $F$ a flat of $M(\mathcal{A}).$ Moreover let $L_F$ be the linear space $\bigcap_{i \in F} H_i$ associated to $F.$ We define 
\[\begin{array}{lll}
\mathcal{A}_F & = & \{ H_i \in \mathcal{A}: L_F \subset H_i \} =\{ H_i \in \mathcal{A}: i \in F \} \\
\mathcal{A}^F & = & \{ H_i \cap L_F: i \notin F \}.\\
\end{array}\]
Then $\mathcal{A}_F$ is an arrangement in $V$ such that $M(\mathcal{A}_F) = M \vert_F$. Moreover, if $M/F$ is simple, then the hyperplanes $H_i \cap L_F$ for $i \notin F$ are pairwise distinct, so that 
$\mathcal{A}^F$ is an arrangement of hyperplanes in $L_F$ satisfying $M(\mathcal{A}^F) = M/F$. \end{example}


\subsection*{Building sets and nested sets} Let $\mathcal{L}(M)$ be the lattice of flats of $M$ with unique minimal element $\hat 0.$ A subset $G \subset \mathcal{L}(M) \backslash \{ \hat 0 \}$ is called a \emph{building set} if, for every $F \in \mathcal{L}(M),$ we have an order-isomorphism $$\left[ \hat 0, F\right]\simeq \prod_{X \in \max(G \cap [\hat 0, F])} \left[ \hat 0, X\right],$$ where, for a set $H \subset \mathcal{L}(M),$ the notation $\max H $ denotes the set of maximal elements.

\begin{example} The set $G_{\min} = \{ F \in \mathcal{L}(M): M \vert_F \textit{ is connected } \}$ is the unique minimal building set, while $G_{\max} = \mathcal{L}(M) \backslash \{\hat 0\}$ is the unique maximal building set of $\mathcal{L}(M).$
\end{example}

For a building set $G \subset \mathcal{L}(M) \backslash \{\hat 0\}$ a subset $S \subset G$ is called a \emph{nested set} for $G$ if for pairwise incomparable elements $S_1, \ldots, S_k,$ and $k \geq 2,$ the closure $\cl(S_1 \cup \cdots \cup S_k) \notin G.$ We denote by $\mathcal{N}(G)$ the set of all nested sets for $G.$

\begin{example} For the maximal building set $G_{\max},$ a subset $S \subset G_{\max}$ is nested if and only if $S$ is a chain of flats in $\mathcal{L}(M) \backslash \{\hat 0\}.$ If the building set is minimal, then a nested set can contain incomparable elements.
\end{example}

\subsection*{Fan structures on $\trop(M)$} For a loopfree matroid $M$ there are several natural polyhedral fan structures on the tropical linear space $\trop(M).$ In \cite{fs}, Feichtner and Sturmfels compare these fan structures, from the coarsest, called the Bergman fan, to the finest, called the fine subdivision of $\trop(M)$. In the following we will define these two fans. \\

Let $M$ be a loopfree matroid on $E(M) = \{0, \ldots, n\}$ and $\{e_0, e_1, \ldots, e_n\}$ be the standard basis of $\mathbb{Z}^{n+1}.$ For $F \subset E(M) = \{0, \ldots, n\}$ let $e_F = \sum_{i \in F} e_i.$ Then we define the \emph{matroid polytope} $P_M$ as $$P_M = \conv(e_B: B \in \mathcal{B}(M)) \subset \mathbb{R}^{n+1},$$ the convex hull of the incidence vectors on the bases of $M.$ Feichtner and Sturmfels prove in \cite[Proposition 2.4]{fs} that the dimension of the matroid polytope $P_M$ equals $n+1 - \kappa(M),$ where $\kappa(M)$ denotes the number of connected components of $M.$ In particular, if $M$ is connected, then $P_M$ has dimension $n.$ Faces of the matroid polytope are themselves matroid polytopes. In fact, if $S$ is a face of $P_M,$ then we define the \emph{degeneration matroid} $M_S$ as the matroid on the ground set $E(M)$ such that $I \subset E(M)$ is an independent set of $M_S$ if and only if there exists a vertex $e_B$ of the face $S$ with $I \subset B.$ Then $S$ coincides with the matroid polytope $P_{M_S}.$

Let  $\mathcal{G}(M)$ be the outer normal fan of the matroid polytope $P_M.$ There is an equivalence relation on vectors in $\mathbb{R}^{n+1}/\mathbb{R}\cdot \textbf{1}$, where $u \sim v$ if and only if $u$ and $v$ achieve their maximum value on the same face of $P_M.$ The equivalence classes form the relative interiors of convex polyhedral cones. For a face $S$ of $P_M,$ we will denote by $\sigma_S$ the cone in $\mathbb{R}^{n+1}/\mathbb{R}\cdot \textbf{1}$ obtained by taking the closure of the equivalence class of vectors attaining their maximum value on the face $S$. 

Finally, the \emph{Bergman fan} $\mathfrak{B}(M)$ is the subfan of the projection of  $\mathcal{G}(M)$ to $\mathbb{R}^{n+1}/\mathbb{R}\cdot \textbf{1}$ consisting of those cones $\sigma_S$ for which the degeneration matroid $M_S$ is loopfree, i.e. the union of its bases is the complete ground set. The support $\vert \mathfrak{B}(M) \vert $ of the Bergman fan is the tropical linear space $\trop(M),$ by \cite[Corollary 4.2.11]{tropbook}. 
If $M$ is connected, then $\trop(M)$ has zero-dimensional lineality space in $\mathbb{R}^{n+1}/\mathbb{R}\cdot \textbf{1}$ by \cite[Lemma 2.3]{FR}.\\

For an arbitrary matroid $M$ on $E(M) = \{0, \ldots, n\}$ and a fixed building set $G \subset \mathcal{L}(M) \backslash \{\hat 0\}$ we define a rational polyhedral fan $\Sigma_G(M)$ as follows. Let $\{e_0, e_1, \ldots, e_n\}$ be the standard basis of $\mathbb{R}^{n+1}.$ For $F \in \mathcal{L}(M)$ we set $v_F= \sum_{i \in F} e_i$ and for each $S \in \mathcal{N}(G)$ we define a cone $\sigma_S = \cone(v_F: F \in S) + \mathbb{R}\cdot \textbf{1}$ in $\mathbb{R}^{n+1}/\mathbb{R} \cdot \textbf{1}.$ Then the \emph{nested set fan} of $M$ with respect to $G$ is defined as $$\Sigma_G(M) = \{ \sigma_S: S \in \mathcal{N}(G)\}.$$ The fan $\Sigma_G(M)$ is a pure simplicial fan of dimension $r(M)-1$ in $\mathbb{R}^{n+1}/ \mathbb{R}\cdot \textbf{1}$ by \cite[Proposition 2]{fy}.
For a loopfree matroid $M$ the support of this fan equals the tropical linear space $\trop(M).$ If $G = G_{\max},$ the nested set fan $\Sigma_G(M)$ is also called the \emph{fine subdivision}. We will denote $\Sigma_G(M)$ in this case by $\Sigma(M).$ While for $G= G_{\min},$ we will write $\Sigma_{\min}(M)$ for the \emph{minimal nested set fan} $\Sigma_G(M).$

Let $F$ and $G$ be two flats of a connected matroid $M,$ such that $F \subset G.$ Then $(E(M) \backslash F) \cap G$ is a flat of the contraction matroid $M/F.$ We denote by $M[F,G]$ the matroid obtained by restricting the contraction $M/F$ to $(E(M) \backslash F)\cap G.$ Then $M[F,G]$ is a matroid of rank $r(G) - r(F).$ The following result due to Feichtner and Sturmfels gives a combinatorial criterion for the fans $\Bfrak(M)$ and $\Sigma_{min}(M)$ to coincide.

\begin{thm}\label{thm:fs}\cite[Proposition 5.3]{fs} For a loopfree matroid the minimal nested set fan $\Sigma_{min}(M)$ equals the Bergman fan $\Bfrak(M)$ if and only if the matroid $M[F,G]$ is connected for every pair of flats $F \subset G$ with $G$ connected.
\end{thm}

\begin{example} 

\smallskip

\noindent (1) For an arrangement $\cA$ of $n$ lines in $\mathbb{P}^2,$ the Bergman fan and the nested set fan coincide in all cases but the following: There exists a line $L$ in $\cA$ and points $a,b \in L$ such that each remaining line passes through $a$ or $b.$ In this case $\Bfrak(M_{\cA}) \neq \Sigma_{min}(M_{\cA}).$ In fact, the wonderful compactification $X_{wnd}(\cA)$ is just the blow-up of $\mathbb{P}^2$ in the two points $a$ and $b,$ while $X_{vc}(\cA)$ is obtained by blowing-down the strict transform of $L$ in $X_{wnd}(\cA),$ hence $X_{vc}(\cA)$ is isomorphic to $\mathbb{P}^1 \times \mathbb{P}^1.$ 

\smallskip

\noindent (2) In \cite[Theorem 1.2]{arw} Ardila, Reiner and Williams prove that the Bergman fan and the minimal nested set fan coincide for finite root system arrangements. In particular, for a finite root system arrangement $\cA$ the compactifications $X_{wnd}(\cA)$ and $X_{vc}(\cA)$ coincide. Thus, we have $X_{vc}(A_n) = X_{wnd}(A_n)$ for the braid arrangement $A_n$ defined in Example 2.1.

\smallskip

\noindent (3) The Deligne-Mumford compactification $\overline{M}_{0,n}$ coincides with the minimal wonderful compactification of the complement of the complex braid arrangement $A_{n-2}$ by \cite[Section 4.3]{DP}.

\end{example}

\section{The span of cones in the Bergman fan}\label{sec:Bergman}

In this section we prove that distinct maximal cones in the Bergman fan of a matroid span distinct linear spaces. This will be useful in the proof of our main theorem.\\

Recall that every cone $\sigma_{\mathcal{F}}$ in the fine subdivision $\Sigma(M)$ is given by a chain of flats
	\[\mathcal{F}: \emptyset \subsetneq F_{1} \subsetneq \dotsb \subsetneq F_{m} \subsetneq F_{m+1} \subset E(M)\,.\]
We will decompose such a chain in the following way: Set $I_{\mathcal{F}}^1 = F_1$ and $I_{\mathcal{F}}^j = F_j \backslash F_{j-1}$ for $j$ in $\{2, \ldots, m+1\}$. Then we can rewrite the chain $\mathcal{F}$ as
	$$\mathcal{F}: \emptyset \subsetneq\ I_{\mathcal{F}}^1\  \subsetneq\  I_{\mathcal{F}}^1 \dot{\cup} I_{\mathcal{F}}^2 \ \subsetneq \ \cdots\  \subsetneq \ I_{\mathcal{F}}^1 \dot{\cup} I_{\mathcal{F}}^2 \dot{\cup} \cdots \dot{\cup} I_{\mathcal{F}}^{m+1} \subset E(M)\,.$$

 \begin{prop}\label{prop:eigenschaft1} Let $M$ be a loopfree matroid on the ground set $E(M)= \{ 0,1, \dotsc, n \}$ of rank $r(M) = r+1$ , and let $\sigma_\mathcal{F}$ and $\sigma_\mathcal{G}$ be two maximal cones in the fine subdivision $\Sigma(M) \subseteq \mathbb{R}^{n+1}/ \mathbb{R} \cdot \textbf{1}$ such that the linear spans $\langle \sigma_\mathcal{F} \rangle$ and $\langle \sigma_\mathcal{G} \rangle$ coincide. Then there exists a cone $\sigma$ in the Bergman fan $\mathfrak{B}(M)$ containing both of them, i.e.  $\sigma_\mathcal{F} \cup \sigma_\mathcal{G} \subseteq \sigma.$
\end{prop}

\begin{proof}  
Assume $\langle \sigma_\mathcal{F} \rangle=\langle\sigma_\mathcal{G} \rangle$ for two maximal cones  $\sigma_\mathcal{F}$ and $\sigma_\mathcal{G}$  in $\Sigma(M)$ given by the chains 
	\[\mathcal{F}: \emptyset \subsetneq F_{1} \subsetneq \dotsb \subsetneq F_{r} \subsetneq F_{r+1}=E(M)\] and

	\[\mathcal{G}: \emptyset \subsetneq G_{1} \subsetneq \dotsb \subsetneq G_{r} \subsetneq G_{r+1}=E(M),\]
respectively. Let us first show that this implies that $\{I_{\mathcal{F}}^j: j = 1, \ldots, r+1\}=\{I_{\mathcal{G}}^j: j = 1, \ldots, r+1\}.$

Recall that the cone $\sigma_{\mathcal{F}}$ associated to a chain $\mathcal{F}$ is defined as $\sigma_{\mathcal{F}} = \cone(v_F: F \in \mathcal{F}) + \mathbb{R} \cdot \textbf{1}.$ Since every maximal chain of flats contains $E(M)= \{0,1, \ldots,n\}$ as the maximal element, we find that
$\langle \sigma_{\mathcal{F}} \rangle = \langle \sigma_{\mathcal{G}} \rangle$ if and only if $\langle \cone(v_F: F \in \mathcal{F}) \rangle = \langle \cone(v_G: G \in \mathcal{G}) \rangle.$

Now, the set of vectors $(x_0, \ldots, x_n)^t \in \langle \cone(v_F: F \in \mathcal{F}) \rangle$ such that $x_l \in \{0,1\}$ for all $l \in \{0, \ldots, n\}$ are sums of incidence vectors 
	\[v_{I_{\mathcal{F}}^j}= \sum_{i \in I_{\mathcal{F}}^j} e_i .\] 

Hence every $v_{I_{\mathcal{F}}^{k}}$ is a sum of suitable $v_{I_{\mathcal{G}}^{j}}$. This sum cannot involve more than one summand since all $v_{I_{\mathcal{G}}^{j}}$ are zero-one-vectors contained in $\langle \cone(v_F: F \in \mathcal{F}) \rangle$ by assumption and $I_{\mathcal{F}}^k \subset I_{\mathcal{F}}^m$ implies $I_{\mathcal{F}}^k = I_{\mathcal{F}}^m.$ Hence we find $\{v_{I_{\mathcal{F}}}^j: j \} = 	 \{v_{I_{\mathcal{G}}}^j: j \}$ which implies our claim.

To each cone $\sigma_{\mathcal{F}}$ we can now associate the set $\mathcal{B}(\sigma_{\mathcal{F}})$ of all subsets of $E(M)$ of the form $\{i_1, \ldots, i_{r+1} : i_j \in I_{\mathcal{F}}^j\}.$
 It is straightforward to show that $\mathcal{B}(\sigma_{\mathcal{F}})$ is in fact a subset of the set of bases $\mathcal{B}(M)$ of $M.$

\medskip

Our next goal is to show that points in the relative interior of $\sigma_{\mathcal{F}}$ achieve their maximum value on the face $S=\conv(e_B: B \in \mathcal{B}(\sigma_{\mathcal{F}}))$ of $P_M.$
For a point $p=(p_0, \ldots, p_n)^t \in \mathbb{R}^{n+1}$ whose image in $\mathbb{R}^{n+1}/\mathbb{R} \cdot \textbf{1}$ lies in the relative interior of $\sigma_{\mathcal{F}}$ the entries $p_i$ and $p_j$ coincide if and only if there exists some $k \in \{ 1, \ldots, r+1\}$ such that $i,j \in I_{\mathcal{F}}^k.$ In particular the point $p$ has $r+1$ distinct entries $x_1, \ldots, x_{r+1}$ satisfying $x_i > x_{i+1},$ and therefore $\langle p,e_{\widetilde{B}} \rangle = \sum_{i=1}^{r+1} x_i$ for all $\widetilde{B} \in \mathcal{B}(\sigma_{\mathcal{F}}).$ Now let $B \in \mathcal{B}(M)$ be an arbitrary basis of the matroid $M$ and $\lambda_j = \vert B \cap I_{\mathcal{F}}^j \vert $ for $j \in \{1, \ldots, r+1\}.$ The sum $\sum_{j=1}^{r+1} \lambda_j$ is equal to $r+1$ since $\vert B \vert = r+1.$ Moreover for each $F_k$ in the maximal chain $\mathcal{F}$ we have that $\vert B \cap F_k \vert \leq k,$ hence $\sum_{j=1}^{k} \lambda_j \leq k$ for all $k \in \{1, \ldots, r+1\}.$ Then $\langle p, e_B \rangle$ has the form
$$\langle p, e_B \rangle = \lambda_1 x_1 + \lambda_2 x_2 + \cdots + \lambda_r x_r. $$ If $\lambda_j=1$ for all $j \in \{1, \ldots, r+1 \}$ then $B$ is an element of $\mathcal{B}(\sigma_{\mathcal{F}})$ and there is nothing to show. Therefore assume that there is a smallest index $i_1$ such that $\lambda_{i_1} > 1.$ Since $\sum_{j=1}^{i_1} \lambda_j \leq i_1$ we can accordingly find an index $j_1<i_1$ such that $\lambda_{j_1} = 0.$ Note that, in particular, the coordinate $x_{j_1} > x_{i_1}.$ We conclude that 
\begin{align*}
\langle p, e_B \rangle &= \lambda_1x_1 + \cdots + \lambda_{j_1}x_{j_1} + \cdots + \lambda_{i_1}x_{i_1} + \cdots + \lambda_{r+1}x_{r+1} \\
&<  \lambda_1x_1 + \cdots + (\lambda_{j_1}+1)x_{j_1} + \cdots + (\lambda_{i_1}-1)x_{i_1} + \cdots + \lambda_{r+1}x_{r+1}
\end{align*} 
Note that the new coefficients in this combination still sum up to $r+1$. After repeating the previous argument finitely many times, we reach the situation where all coefficients are smaller or equal to $1$, and hence equal to $1$. Therefore
	\[\langle p, e_B \rangle < x_1 + \cdots + x_{r+1} = \langle p, e_{\widetilde{B}} \rangle\,.\] In particular, the point $p$ achieves its maximal value indeed on the face $S=\conv(e_B: B \in \mathcal{B}(\sigma_{\mathcal{F}}))$ of $P_M.$
	
\medskip
	
We can now prove the theorem. 
Let $\sigma_{\mathcal{F}}$ and $\sigma_{\mathcal{G}}$ be maximal cones in $\Sigma(M)$ associated to the maximal chains of flats $\mathcal{F}$ and $\mathcal{G}$ such that $\langle \sigma_{\mathcal{F}} \rangle = \langle \sigma_{\mathcal{G}} \rangle.$ Then we have seen that the sets $\{I_{\mathcal{F}}^j: j = 1, \ldots, r+1\}$ and $\{I_{\mathcal{G}}^j: j = 1, \ldots, r+1\}$ in the decompositions of the chains $\mathcal{F}$ and $\mathcal{G}$ coincide. In particular the set of bases $\mathcal{B}(\sigma_{\mathcal{F}})$ and $\mathcal{B}(\sigma_{\mathcal{G}})$ are equal. Therefore the points in the relative interiors of $\sigma_{\mathcal{F}}$ and $\sigma_{\mathcal{G}}$ take their maximum value on the same face $S$ of the matroid polytope $P_M$ which is spanned by the incidence vectors $e_B,$ where $B$ runs over the bases in $\mathcal{B}(\sigma_{\mathcal{F}}) = \mathcal{B}(\sigma_{\mathcal{G}}).$ Hence $\sigma_{\mathcal{F}} \cup \sigma_{\mathcal{G}} \subset \sigma_S,$ where $\sigma_S$ denotes the closure of the equivalence class of points achieving their maximum value on $S.$ Recall that the Bergman fan $\Bfrak(M)$ consists of those cones $\sigma_S$ such that the degeneration matroid $M_S$ is loopfree. Since $\bigcup_{j=1}^{r+1} I_{\mathcal{F}}^j = E(M),$ the cone $\sigma_S$ is in fact a cone of the Bergman fan.
\end{proof}

In \cite{rtw} it is shown for the full rational arrangement $\cA$  in projective space over a finite field that every automorphism of $\Omega_\cA$ (which is Drinfeld's half-space over the finite field) extends to an automorphism of its wonderful compactification $X_{wnd}(\cA).$ One key point  in this argument is \cite[Lemma 2.2]{rtw} which states that the image of the skeleton of the analytification of $\Omega_\cA$ under a natural toroidal embedding has the property that distinct cones span distinct linear spaces. We will now show an analogous property for the Bergman fan of loopfree matroids. 

\begin{thm}\label{thm:cones} Let $M$ be a loopfree matroid and $\mathfrak{B}(M)$ its Bergman fan. Then distinct maximal cones of $\mathfrak{B}(M)$ span distinct linear spaces.
\end{thm}

\begin{proof}
Suppose there are distinct maximal cones $\sigma_1$ and $\sigma_2$ in the Bergman fan $\mathfrak{B}(M)$ such that $\langle \sigma_1 \rangle = \langle \sigma_2 \rangle$. Since the fine subdivision refines the Bergman fan, there are maximal cones $a_1$ and $a_2$ in $\Sigma(M)$ such that $a_1 \subseteq \sigma_1$ and $a_2 \subseteq \sigma_2.$ In particular $\dim a_1 = \dim \sigma_1$ and $\dim a_2 = \dim \sigma_2$ and hence $\langle a_1 \rangle = \langle \sigma_1 \rangle$ and $\langle a_2 \rangle = \langle \sigma_2 \rangle.$ Since the linear hulls of $\sigma_1$ and $\sigma_2$ coincide, so do the linear hulls of $a_1$ and $a_2.$ Therefore the cones $a_1$ and $a_2$ are maximal cones in the fine subdivision $\Sigma(M)$ whose linear hulls coincide.  By \autoref{prop:eigenschaft1} there exists a cone $\sigma$ in the Bergman fan $\mathfrak{B}(M)$ which contains the union $a_1 \cup a_2.$ In particular $a_1 \subseteq \sigma_1 \cap \sigma$ and $a_2 \subseteq \sigma_2 \cap \sigma.$ That means $\sigma$ and $\sigma_i$ are cones in the Bergman fan which intersect in full dimension. Hence $\sigma_1 = \sigma = \sigma_2.$ 
\end{proof}

\section{Extending morphisms between arrangement complements}\label{sec:extend}

 Let 
$\cA$ be a finite arrangement of hyperplanes in projective space $\mathbb{P}(V)$, where $V$ is a vector space of dimension $d+1$ over an arbitrary ground field $K$.  We will now show our main result on extension of endomorphisms of $\Omega_\cA$, using the notation from \autoref{sec:background}.

\begin{thm}\label{mainthm}
(i) If the hyperplane arrangement $\cA$ is essential and connected, then every dominant morphism  $f: \Omega_\cA \rightarrow \Omega_\cA$ can be extended to a morphism on the visible contour compactification $\overline{f}: X_{vc}(\cA) \rightarrow X_{vc}(\cA)$. 

(ii) If $\cA$ is essential and connected, then every automorphism of $\Omega_\cA$ extends to an automorphism of its visible contour compactification $X_{vc}(\cA).$
\end{thm}

\begin{proof}
{\bf(i)} Let $f: \Omega_\cA \to \Omega_\cA$ be a dominant endomorphism and denote by $j: \Omega_\cA \hookrightarrow T$ the natural embedding of $\Omega_\cA$ into its intrinsic torus as in section \ref{sec:background}. Let $l_0, \ldots, l_n$ be the linear forms associated to the hyperplanes in $\cA$, and let $x_1, \ldots, x_n$ be the basis of the character group $N^\ast$ of $T$ such that $j$ is given by $x_i \mapsto l_i/l_0$. 

Since $f$ is dominant, the associated map on coordinate rings $f^\sharp: \mathcal{O}(\Omega_\cA) \to \mathcal{O}(\Omega_\cA)$ is injective. 
As the  multiplicative group $\mathcal{O}(\Omega_\cA)^*/K^*$ is generated by the classes $\left[\frac{l_i}{l_0}\right]$ of $\frac{l_i}{l_0}$ for $i$ in $\{ 1, \ldots, n\}$ we find that $ f^\sharp \left(\frac{l_i}{l_0} \right)  = \lambda_i \prod_{j=1}^{n} \left(\frac{l_i}{l_0}\right)^{a_{ij}}$ for integers $a_{ij}$ and $\lambda_i \in K^\ast$. In terms of the coordinates $x_1,\ldots, x_n$,  the matrix $A = (a_{ij})$ defines an endomorphism of $N^\ast$. Hence it gives rise to a torus homomorphism $h: T \rightarrow T$.  Let $\lambda$ be the $K$-rational point of $T$ with coordinates $(\lambda_1, \ldots, \lambda_n)$, and write $t_\lambda: T \rightarrow T$ for translation by $\lambda$. 

Putting $g = t_\lambda \circ h$, we have a commutative diagram
\[
\xymatrix{\Omega_\cA  \ar[d]^{f}\ \ar[r] & T \ar[d]^g\\
\Omega_\cA \ar[r] &  T  \ }\]

We claim that $g$ extends to a morphism of the toric variety $Y_{\Bfrak(M_\cA)}.$ In fact, since the natural action of the intrinsic torus on itself extends to $Y_{\Bfrak(M_\cA)},$ so does the automorphism $t_\lambda$. Therefore we only need  show that the group homomorphism $h$ extends as well. In order to prove this, we will show that the linear endomorphism $a$ induced by $A^t$ on the cocharacter space $N_\mathbb{R}$ is compatible with the fan $\Bfrak(M_\cA)$ and hence gives rise to a toric morphism $\overline{g}: Y_{\Bfrak(M_\cA)} \to Y_{\Bfrak(M_\cA)}$ extending $g$. \\

The linear map $a: N_{\mathbb{R}} \rightarrow N_{\mathbb{R}} $ maps $\trop(\Omega_\cA)$ to $\trop(\Omega_\cA)$. By \cite[Proposition 3.1]{Tev}, this is a surjection $a: \trop(\Omega_\cA) \to \trop(\Omega_\cA).$ In order to show that $a$ is compatible with the fan structure on $\Bfrak(M_\cA),$ we first note that all vectors of the standard basis $e_1, \ldots, e_n$ of $N_{\mathbb{R}}$ lie in $\trop(\Omega_\cA).$ Since $a$ is surjective on $\trop(\Omega_\cA)$ the linear map $a$ is an element of $\GL_n(\mathbb{R})$. In fact, it lies in $\GL_n(\mathbb{Z})$ if and only if the dominant map $f$ is an automorphism of $\Omega_\cA$, since in the latter case we can apply the previous reasoning to $f^{-1}$. 

For a maximal cone $\sigma$ in $\Bfrak(M_\cA)$ of dimension $\dim(\sigma) = \dim(\Omega_\cA) = d,$ the image $a(\sigma)$ is again a $d-$dimensional cone in $N_{\mathbb{R}}.$ We need to show that $a(\sigma)$ is contained in a cone of $\Bfrak(M_\cA).$ Suppose $a(\sigma)$ intersects two different maximal cones $\sigma_1$ and $\sigma_2$ in $\Bfrak(M_\cA)$ in their relative interiors. Since $d = \dim(a(\sigma)) = \dim(\sigma_1) = \dim(\sigma_2)$ it follows that the linear hulls spanned by $\sigma_1$ and $\sigma_2$ coincide. This is a contradiction to \autoref{thm:cones}. In particular $a$ is a homomorphism of fans and hence $h$ extends to a morphism $\overline{h}: Y_{\Bfrak(M_\cA)} \to Y_{\Bfrak(M_\cA)}.$ Therefore also $g$ extends to a morphism $\overline{g}: Y_{\Bfrak(M_\cA)} \to Y_{\Bfrak(M_\cA)}.$ By restricting $\overline{g}$ to the closure of $\Omega_\cA$ in the toric variety $Y_{\Bfrak(M_\cA)}$ we get the required extension.

{\bf(ii)} follows by applying (i) to $f$ and to its inverse morphism. Note that since $X_{vc}(\cA)$ is reduced and separated, endomorphisms of $\Omega_\cA$ extend uniquely. \end{proof}

Note that in the proof of the previous theorem we could also have argued with the fact that the Bergman fan is the coarsest fan structure on $\trop(\Omega_\cA)$.  The argument given here is more intrinsic, and we hope that \autoref{thm:cones} is also useful for other purposes.

\begin{corollary}\label{finite} If the hyperplane arrangement $\cA$ is essential and connected, every dominant morphism $f: \Omega_\cA \rightarrow \Omega_\cA$ is finite.
\end{corollary}

\begin{proof}
Let $f: \Omega_\cA \to \Omega_\cA$ be a dominant endomorphism. Then by  \autoref{mainthm}, $f$ extends to a morphism $\overline{f}: X_{vc}(\cA) \to X_{vc}(\cA)$ and this yields the following Cartesian diagram:
\[
\xymatrix{\Omega_\cA  \ar[d]^{f}\ \ar[r]  &X_{vc}(\cA)  \ar[d]^{\overline{f}}\\
\Omega_\cA \ar[r] &  X_{vc}(\cA) }\]

Since $\overline{f}$ is proper, so is $f.$ As  a proper morphism of affine varieties $f$ is indeed finite.
\end{proof}

In view of the previous \autoref{thm:fs}  by Feichtner and Sturmfels, we also have the following corollary. 

\begin{corollary} \label{cor:wond}Assume that $\cA$ is essential and connected, and that $M[F,G]$ is connected for every pair of flats $F$ and $G$ in $\mathcal{L}(M_\cA)$ with $G$ connected and $F \subset G.$ Then every dominant morphism $f:  \Omega_\cA \rightarrow \Omega_\cA$ extends to a morphism  on the wonderful compactification $\ov{f}: X_{wnd}(\cA) \rightarrow X_{wnd}(\cA)$. \end{corollary}

An important example where the conditions of this corollary are fulfilled is the following one. 

\begin{corollary} \label{cor:drinfeld} Assume that $K = \mathbb{F}_q$ is a finite field and $\cA$ is the  arrangement consisting of all $\mathbb{F}_q$-rational hyperplanes in $\mathbb{P}^n_K$, so that $\Omega_\cA$ is Drinfeld's  half-space over $\mathbb{F}_q$. 

(i) Every dominant morphism $f: \Omega_\cA \rightarrow \Omega_\cA$ extends to a morphism $\ov{f}: X_{wnd}(\cA) \rightarrow X_{wnd}(\cA) $ on the wonderful compactification. In particular, $f$ is a finite morphism.

(ii) Every automorphism $f: \Omega_\cA \to \Omega_\cA$ extends to an automorphism $\overline{f}: X_{wnd}(\cA) \to X_{wnd}(\cA).$
\end{corollary}

\begin{proof}

Let $M_\cA$ be the matroid associated to $\cA$ and $F$ and $G$ two flats of $M,$ such that $F \subset G.$ Then $M[F,G]$ is the matroid associated to the full arrangement in $\mathbb{P}^{r(G)-r(F)-1},$ where $r(F)$ and $r(G)$ denotes the rank of the flats $F$ and $G$, respectively. In particular, $M[F,G]$ is connected for all pairs of flats $F \subset G$ of $M_\cA.$ Therefore in the case of the Drinfeld's half-space the fans $\Bfrak(M_\cA)$ and $\Sigma_{min}(M_\cA)$ coincide by \autoref{thm:fs} and, in particular, so do the visible contour and the wonderful compactification. Hence our claims are a direct consequence of \autoref{mainthm} and \autoref{finite}.
\end{proof}

In particular, \autoref{cor:drinfeld} gives an alternative proof of \cite[Proposition 2.1]{rtw} without using analytic geometry. Moreover, it generalizes this result to a large class of arrangement complements.

\end{document}